\documentclass[11pt,a4paper]{article}
\usepackage[margin=1in]{geometry}
\usepackage[T1]{fontenc}
\usepackage[utf8]{inputenc}
\usepackage{tikz}
\usepackage{hyperref}
\usepackage{amsmath,amssymb,amsfonts,amsthm,amscd, bm, mathtools}
\makeatletter
\newtheorem*{rep@theorem}{\rep@title}
\newcommand{\newreptheorem}[2]{%
\newenvironment{rep#1}[1]{%
 \def\rep@title{#2 \ref{##1}}%
 \begin{rep@theorem}}%
 {\end{rep@theorem}}}
\makeatother
\newtheorem{theorem}{Theorem}[section]
\newreptheorem{theorem}{Theorem}

\newtheorem{definition}[theorem]{Definition}
\newtheorem{lemma}[theorem]{Lemma}
\newtheorem{corollary}[theorem]{Corollary}

\title{Edge-connectivity between edge-ends of infinite graphs}
\author{Leandro Aurichi$^*$ and Lucas Real\footnote{Institute of Mathematics and Computer Sciences, University of S\~{a}o Paulo, S\~{a}o Paulo, Brazil}}

\date{}

\begin{document}
\maketitle
\begin{abstract}
    In infinite graph theory, the notion of \textit{ends}, first introduced by Freudenthal and Jung for locally finite graphs, plays an important role when generalizing statements from finite graphs to infinite ones. Nash-Williams' Tree-Packing Theorem and MacLane's Planarity Criteria are examples of results that allow a topological approach, in which ends might be considered as \textit{endpoints} of rays. In fact, there are extensive works in the literature showing that classical (vertex-)connectivity theorems for finite graphs can be discussed regarding ends, in a more general context. However, aiming to generalize results of edge-connectivity, this paper recalls the definition of \textit{edge-ends} in infinite graphs due to Hahn, Laviolette and \v{S}irá\v{n}. In terms of that object, we state an edge version of Menger's Theorem (following a previous work of Polat) and generalize the Lovász-Cherkassky Theorem for infinite graphs with edge-ends (inspired by a recent paper of Jacobs, Joó, Knappe, Kurkofka and Melcher).     
\end{abstract}

\section{Introduction}

Most graphs in this paper are \textit{simple}, in the sense that loops and multiple edges are not considered. Exceptionally, graphs that might admit parallel edges are referred as \textit{multigraphs}. If $S$ and $H$ are subgraphs of a given graph $G$, we define the neighborhoods $N(S) = \{v \in V(G)\setminus V(S) : v \text{ has a neighbor in }S\}$ and $N_H(S):= N(S)\cap V(H)$. If $G$ is infinite, we recall that a \textbf{ray} is a one-way infinite path within it. In other words, it is a subgraph of the form $r = v_0v_1v_2\dots$, in which $v_i \in V(G)$ and $v_iv_{i+1}\in E(G)$ for every $i \in \mathbb{N}$. Thus, we say that $r$ \textbf{starts} at $v_0$ and that any infinite connected subgraph of $r$ is its \textbf{tail}. Similarly, a \textbf{double ray} is a subgraph of the form $R = \dots v_{-2}v_{-1}v_0v_1v_2\dots$, in which $\{v_iv_{i+1}\}_{i\in \mathbb{Z}}\subseteq E(G)$. For any $n \in \mathbb{Z}$, the subgraphs $r_+ = v_{n}v_{n+1}v_{n+2}\dots$ and $r_- = v_{n}v_{n-1}v_{n-2}\dots$ are called the \textbf{half-rays} of $R$. Intuitively, a ray describes a direction in the ambient graph $G$, which is formalized by the following equivalence relation: we write $r\sim s$ whenever $r$ and $s$ cannot be separated by finitely many vertices, namely, the tails of $r$ and $s$ belong to the same connected component of $G\setminus S$ for every finite set $S\subset V(G)$. If this is not the case, we say that some finite set $S\subset V(G)$ \textbf{separates} $r$ and $s$.

Equivalently, $r\sim s$ if and only if there is an infinite family of $r-s$ disjoint paths. In this notation, given $A$ and $B$ two vertex subsets of $G$, we say that an $A-B$ \textbf{path} is a finite path $v_0v_1v_2\dots v_n$ with precisely one endpoint in $A$ and the other in $B$, i.e., $v_0\in A$, $v_n \in B$ and $v_1,v_2,\dots,v_{n-1}\in V(G)\setminus (A\cup B)$. Hence, it is easily seen that $\sim$ is an equivalence relation over $\mathcal{R}(G)$, the collection of rays of $G$. The quotient $\frac{\mathcal{R}(G)}{\sim}$, denoted by $\Omega(G)$, is called the \textbf{end space} of $G$. An element $[r]\in \Omega(G)$, in its turn, is the \textbf{end} of the ray $r$.

This notion was first approached by Freudenthal in \cite{freudenthal} and by Hopf in \cite{hopf}, within algebraic discussions related to group theory. Halin in \cite{halin}, however, properly stated this definition for infinite graphs, inspiring generalizations of results from finite graph theory since then. For example, Stein in \cite{stein} showed how ends play an important role when extending the Nash-Williams' Three Packing Theorem for locally finite graphs. Her joint work with Bruhn in \cite{bruhn} presents a similar discussion, in which MacLane's Planarity Criteria is stated for locally finite graphs with their ends. Although this paper works under weaker conditions, local finiteness is a convenient hypothesis when preserving theorems from finite graph theory. As pointed out by Diestel in his survey \cite{survey}, the ends are helpful, and occasionally unavoidable, structures to that aim.

In any case, the results mentioned in the previous paragraph are obtained throughout a topological approach. Indeed, from the definition of $\sim$, the connected components obtained after the removal of finite vertex sets suggest a natural topology for $\hat{V}(G) := V(G)\cup \Omega(G)$. More precisely, given $S\subset V(G)$ finite and $r \in \mathcal{R}(G)$, let $C(S,[r])$ be the set of vertices that belong to the connected component of $G\setminus S$ in which $r$ has a tail. By defining $$\Omega(S,[r]) = \{[s] \in \Omega(G) : S\text{ does not separate }r \text{ and }s\},$$ we declare $S$ and $\hat{V}(S,[r]):= C(S,[r])\cup \Omega(S,[r])$ as basic open sets in $\hat{V}(G)$. Since equivalent rays are infinitely connected, it is easily seen that $\Omega(S,[r])$ is well-defined.

With this topology for $\hat{V}(G)$, the end $[r]$ is precisely the boundary of the ray $r$ (as a set of vertices), formalizing the idea that $r$ converges to $[r]$ or that $[r]$ might be seen as an endpoint of $r$. Then, roughly speaking, in the previously mentioned generalizations of results from finite graph theory, the ends of $G$ play the same role as \textit{vertices}, but ``at the infinity''.

On the other hand, the literature related to the \textit{edge}-connectivity of infinity graphs and their ends is still not broad, with few references on the subject. An exception, for instance, is the paper \cite{jacobs}, in which Jacobs, Joó, Knappe, Kurkofka and Melcher generalized the classical Lovász-Cherkassky Theorem (recalled at the end of the next section) to locally finite graphs with ends. Their proof relies on the version of the same theorem for countable graphs, obtained by Joó in \cite{joo} without mentioning ends. Actually, the Lovász-Cherkassky Theorem is now proven for arbitrary infinite graphs in \cite{joogeral}, where Joó extended his previous work. In this paper, we will revisit all these results by considering infinite graphs and their edge-ends. As a conclusion, at the end of Section \ref{sec-paths} we establish the following:

\begin{theorem}[Lovász-Cherkassky for infinite graphs with edge-ends]\label{t1}
	Let $G = (V,E)$ be any infinite graph and fix a discrete subspace $T\subset V(G)\cup \Omega_E(G)$ of $\hat{E}(G)$. Suppose that $|\delta(X)|$ is even or infinite for every $X\subset V(G)$ in whose closure $T$ lies. Then, there is a collection $\mathcal{P}$ of edge-disjoint $T-$paths such that, for every $t\in T$, there is a cut separating $t$ from $T\setminus \{t\}$ that lies on the family $\mathcal{P}_t = \{P\in \mathcal{P}: t \text{ is an endpoint of }P\}$.
\end{theorem}

Especially regarding connecting paths, the definitions and notations employed in the above statement are presented throughout the next sections. However, we remark that the notion of edge-ends introduced by Theorem \ref{t1} is inspired by the definitions of ends, but slightly modified so that edge-connectivity results can be approached more properly. Thus, given rays $r,s\in \mathcal{R}(G)$, we now say that $r$ and $s$ are \textbf{edge-equivalent}, writing $r\sim_Es$, if, and only if, the tails of $r$ and $s$ belong to the same connected component of $G\setminus F$, for every finite set of \textit{edges} $F\subset E(G)$. Equivalently, there is an infinite family of edge-disjoint paths, with distinct endpoints, connecting $r$ and $s$. If this is not the case, we say that a finite set $F\subset E(G)$ \textbf{separates} $r$ and $s$. When $F$ can be described by $\{uv \in E(G): u \in X, v \in V(G)\setminus X\}$ for some $X\subset V(G)$, we write $F = \delta(X)$ and call $F$ the \textbf{cut} corresponding to the vertex set $X$.

As it is easily seen, $\sim_E$ is also an equivalence relation over $\mathcal{R}(G)$, so that the quotient $\frac{\mathcal{R}(G)}{\sim_E}$ defines the \textbf{edge-end space} of $G$, denoted by $\Omega_E(G)$. Similarly, we denote the equivalence class of ray $r$ by $[r]_E$ and we call it an \textbf{edge-end} of the graph. We even observe that $\sim$ and $\sim_E$ are the same relation if $G$ is locally finite, because, in this case, rays which are separated by finitely many vertices can be separated by the finitely many edges incident to them. In particular, the main result of \cite{jacobs} will be preserved by Theorem \ref{t1}. When there are vertices of infinite degree, however, $\sim_E$ may identify more rays, as exemplified by Figure \ref{raioduplodominado}.

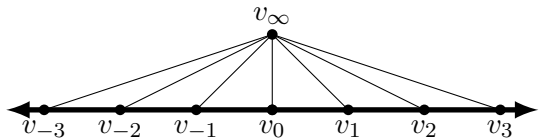
\begin{figure}[ht]
	\centering
	\begin{tikzpicture}
		\draw[line width=2pt, black, -latex] (0,0) -- (3.5,0);
		\draw[line width=2pt, black, -latex] (0,0) -- (-3.5,0);
		
		\foreach \x in {-3,-2,-1,0,1,2,3}
		\fill (\x,0) circle (2pt) node[below]{$v_{\x}$};
		
		\fill (0,1) circle (2pt) node[above]{$v_{\infty}$};
		
		\foreach \x in {-3,-2,-1,0,1,2,3}
		\draw (\x,0) -- (0,1);
		
		\fill (0,0) circle (2pt);
	\end{tikzpicture}
	
	\caption{This graph has two ends, corresponding to the rays $v_0v_1v_2\dots$ and $v_0v_{-1}v_{-2}\dots$, for example. However, it has only one edge-end.}
	\label{raioduplodominado}
\end{figure}

In the literature, the definition of $\sim_E$ was first presented by the paper \cite{edgeends} of Hahn, Laviolette and \v{S}irá\v{n}. Besides this article, the discussions carried out by Georgakopoulos in \cite{georgakopoulos} comprise one of the few other works which approach the notion of edge-ends. In fact, since $\sim$ and $\sim_E$ are the same equivalence relation over locally finite graphs, most references in infinite graph theory may not mention the latter one in their studies.

Despite that, we observe that Theorem \ref{t1} also requires topological conditions on $V(G)\cup \Omega_E(G)$. Indeed, once presented the topology for $\hat{V}(G)$, the relation $\sim_E$ suggests a topology for $\hat{E}(G) = V(G)\cup \Omega_E(G)$ as well. Now, given a finite $F\subset E(G)$ and an edge-end $[r]_E\in \Omega_E(G)$, we denote by $C_E(F,[r]_E)$ the (vertices of the) connected component of $G\setminus F$ in which $r$ has a tail. Analogously, if $v\in V(G)$ is a vertex, $C_E(F,v)$ denotes the vertex set of the connected component in $G\setminus F$ containing it. As a concept handled by Theorem \ref{t1} and that will be often used throughout the paper, we now say that $F$ \textbf{separates} two subsets $A,B\subseteq \hat{E}(G)$ if $C_{E}(F,x)\neq C_{E}(F,y)$ for every $x\in A$ and $y\in B$. In its turn, since $C_E(F,[r]_E) = C_E(F,[r']_E)$ if $r$ and $r'$ are edge-equivalent rays, the following set is also well-defined: $$\Omega_E(F,[r]_E) = \{[s]_E \in \Omega_E(G) : F \text{ does not separate }r\text{ and }s\}.$$

Then, by denoting $\hat{E}(F,[r]_E) := C_E(F,[r]_E)\cup \Omega_E(F,[r]_E)$, the aimed topology for $\hat{E}(G)$ has $\{\{v\} : v \in V(G)\}\cup \{\hat{E}(F,[r]_E): F\subset E(G)\text{ finite}, [r]_E\in \Omega_E(G)\}$ as a basis. In particular, the subspace $\Omega_E(G)\subset \hat{E}(G)$ has $\{\Omega_E(F,[r]_E) : [r]_E\in \Omega_E(G), F\subset E(G)\text{ finite}\}$ as a basis. Thus, also considering $\Omega(G)\subset \hat{V}(G)$ with the corresponding inherited structure, a comparison between the classes of topological spaces $\Omega:=\{\Omega(G) : G \text{ graph}\}$ and $\Omega_E :=\{\Omega_E(G) : G \text{ graph}\}$ can be found in \cite{compaulo}. Among other studies in that paper, written in a joint work with Paulo Magalhães Júnior, we discuss how $\Omega_E$ is a proper subfamily of $\Omega$.

Finally, we revisit the work of Polat in \cite{polat} in order to illustrate how the definition of $\sim_E$ is appropriate when studying edge-connectivity properties. More precisely, throughout the next sections we prove the following Mengerian result:

\begin{theorem}[Menger's Theorem for edge-ends]\label{t2}
	Let $A,B\subset \Omega_E(G)$ be sets of edge-ends such that $\overline{A}\cap B = \emptyset = A\cap \overline{B}$. Then, there is a family $\mathcal{P}$ of edge-disjoint $A-B$ paths and an $A-B$ separator $F\subseteq E(G)$ that lies on it.
\end{theorem}

\section{Suitable connecting paths}
\label{sec-paths}
In this section, we formalize the definitions of $T-$paths and $A-B$ paths employed by Theorems \ref{t1} and \ref{t2} respectively, besides also proving a first common instance of these statements. To this aim, we first recall that Menger's classical theorem for finite graphs is a duality result regarding connectivity. It claims that the minimum amount of vertices in a graph $G$ that separate two subsets of $V(G)$ is attained by the maximum size of a family of disjoint paths connecting them. In this case, each path of this family must contain precisely one vertex from the minimum separator. It was conjectured by Erd\H{o}s that this property is preserved for infinite graphs, which was only verified by Aharoni and Berger in \cite{erdosmenger} after more than thirty years:

\begin{theorem}[\cite{erdosmenger}, Erd\H{o}s-Menger Theorem]\label{erdosmenger}
	Let $G$ be a graph and fix $A,B\subset V(G)$. Then, there exists a family $\mathcal{P}$ of disjoint $A-B$ paths and  an $A-B$ separator $S\subseteq V(G)$ that lies on it.
\end{theorem}

In the above formulation, an $A-B$ \textbf{separator} is a vertex set $S\subset V(G)$ such that there is no $A-B$ path in $G\setminus S$. When saying that $S$ \textbf{lies on} a family $\mathcal{P}$ of disjoint paths, we mean that $S$ is obtained by the choice of precisely one vertex from each element of $\mathcal{P}$.

However, we can rephrase Theorem \ref{erdosmenger} to consider $\mathcal{P}$ as a family of \textit{edge-disjoint} $A-B$ paths, rather than analyzing just disjoint ones. In this case, the minimum separator described by the statement must be a set of edges. Then, we say that an edge-set $F$ in a graph $G$ \textbf{lies on} a family of edge-disjoint paths $\mathcal{P}$ if $F$ is obtained by the choice of precisely one edge from each element of $\mathcal{P}$. In these terms, the result below is folklore, but its proof is presented here for the sake of completeness. Roughly speaking, it is obtained after blowing up vertices to complete graphs and applying the Erd\H{o}s-Menger Theorem:

\begin{corollary}[Erd\H{o}s-Menger Theorem for edges]\label{edgesmenger}
	Let $G$ be a graph and fix disjoint subsets $A,B\subset V(G)$. Then, there exist a family $\mathcal{P}$ of edge-disjoint $A-B$ paths and a cut $\delta(X)$ lying on it such that $A\subset X$ and $B\subset V(G)\setminus X$.  
\end{corollary}
\begin{proof}
	We will define an auxiliary graph $\tilde{G}$. For every $v\in V(G)$, let $K_v$ be a complete graph on $d(v)$ vertices. Then, the vertex set of $\tilde{G}$ will be the disjoint union $\displaystyle \bigcup_{v\in V(G)}V(K_v)$. For every edge $uv \in E(G)$, we define an edge $u'v'$ between the cliques $K_u$ and $K_v$, referred as an \textit{old edge}. Since $|K_v| = d(v)$, we can assume that every vertex of $K_v$ is an endpoint of at most one old edge.

	Apply Theorem \ref{erdosmenger} in order to separate the (disjoint) vertex sets $\tilde{A} = \displaystyle \bigcup_{v \in A}V(K_v)$ and $\tilde{B} = \displaystyle \bigcup_{u\in B}V(K_u)$, fixing the family $\tilde{\mathcal{P}}$ of disjoint $\tilde{A}-\tilde{B}$ paths and the separator $\tilde{S}$ lying on it. For every $\tilde{P}\in \tilde{\mathcal{P}}$ and every $v\in V(G)$, we can assume that $|\tilde{P}\cap V(K_v)| \leq 2$. In fact, if $\tilde{v}_1,\tilde{v}_2 \in V(K_v)$ are non-adjacent vertices in $\tilde{P}$, consider the path $\tilde{P}'$ obtained from $\tilde{P}$ after replacing the subpath connecting $\tilde{v}_1-\tilde{v}_2$ by the edge $\tilde{v}_1\tilde{v}_2$. Being an $\tilde{A}-\tilde{B}$ path, $\tilde{P}'$ must meet $\tilde{S}$ in the vertex of $\tilde{P}\cap \tilde{S}$, since $\tilde{\mathcal{P}}$ is composed by disjoint paths. Therefore, $(\tilde{\mathcal{P}}\setminus \{\tilde{P}\})\cup \{\tilde{P}'\}$ is also a family of disjoint $\tilde{A}-\tilde{B}$ paths on which $\tilde{S}$ lies.

	By considering that $|\tilde{P}\cap V(K_v)| \leq 2$ for every $v\in V(G)$, a path $P$ in $G$ arises from $\tilde{P}\in \tilde{\mathcal{P}}$ after contracting the cliques $\{K_v : v \in V(G)\}$ to their original vertices. Then, $\mathcal{P} = \{P: \tilde{P} \in \tilde{\mathcal{P}}\}$ is a family of edge-disjoint $A-B$ paths. Moreover, each vertex $\tilde{v}\in \tilde{S}$ belongs to a clique of the form $K_v$ and it is the endpoint of a unique old edge $\theta(\tilde{v})$, originally incident in $v\in V(G)$. Note also that $\theta(\tilde{v})$ belongs to the path of $\tilde{\mathcal{P}}$ that contains $\tilde{v}$.

	We observe that every $A-B$ path $Q$ in $G$ must pass through an edge from $\{\theta(\tilde{v}): \tilde{v}\in \tilde{S}\}$. Otherwise, a minimal path in $\tilde{G}$ containing all the old edges of $Q$ will not intersect $\tilde{S}$, contradicting the fact that $\tilde{S}$ separates $\tilde{A}$ and $\tilde{B}$. Therefore, $F = \{\theta(\tilde{v}) : \tilde{v}\in \tilde{S}\}$ is an edge set lying on $\mathcal{P}$ for which there is no $A-B$ path in $G\setminus F$. Although it is no difficult to see that $F = \delta(X)$ for some $X\subset V(G)$ such that $A\subset X$ and $B\subset V(G)\setminus X$, this will later follow by Lemma \ref{cuts}.
\end{proof}

In addition, there also exist statements for Mengerian theorems when considering infinite graphs and their ends, which we shall revisit in order to conclude both Theorems \ref{t1} and \ref{t2}. To that aim, the notions of $A-B$ paths and $A-B$ separators need to be better discussed when $A$ and $B$ are not only vertex subsets, but might contain ends. Nevertheless, when dealing with the topological space $\hat{V}(G)$, different references in the literature present distinct definitions for connecting paths between ends and the corresponding separators, although all the formulations are coincident for locally finite graphs.

Before comparing these concepts, it is useful to distinguish the role played by some vertices of infinite degree. More precisely, fixed a graph $G$, we say that $v\in V(G)$ \textbf{dominates} a ray $r \in \mathcal{R}(G)$ if it is infinitely connected to $r$, in the following sense: for every finite set $S\subset V(G)\setminus \{v\} $, a tail of $r$ and the vertex $v$ belong to the same connected component of $G\setminus S$, i.e., $v\in C(S,[r])$. Equivalently, there is an infinite family of paths connecting $v$ and $r$, pairwise intersecting precisely at $v$. In this case, $v$ dominates any other ray equivalent to $r$, allowing us to say that $v$ dominates the \textit{end} $[r]$. Then, combining notations from \cite{polat} and \cite{jacobs}, we set the following definitions:

\begin{definition}[Connecting paths]\label{paths}
	Let $G$ be a graph. Depending on whether we are considering connectivities of the form ``vertex-vertex'', ``vertex-end'' and ``end-end'', we define the \textbf{paths} and the \textbf{graphic paths} as follows:
	\begin{itemize}
		\item \texttt{Connectivity between vertices:} For vertices $u,v \in V(G)$, a $u-v$ path is any finite path in $G$ which has $u$ and $v$ as endpoints;
		\item \texttt{Connectivity between vertices and ends:} Fix $v\in V(G)$ and $\omega \in \Omega(G)$. A \textbf{graphic} $v-\omega$ \textbf{path} is a ray starting at $v$ whose end is $\omega$. A $v-\omega$ \textbf{path} is either a graphic $v-\omega$ path or a $v-u$ path, in which $u\in V(G)$ is a vertex that dominates $\omega$;
		\item \texttt{Connectivity between ends:} Fix $\omega_1,\omega_2 \in \Omega(G)$. A \textbf{graphic} $\omega_1-\omega_2$ \textbf{path} is a double ray in which one half-ray has end $\omega_1$ and the other has end $\omega_2$. In its turn, a $\omega_1-\omega_2$ \textbf{path} is either a graphic $\omega_1-\omega_2$ path or a $v-\omega_i$ path, for some $i\in \{1,2\}$ and a vertex $v\in V(G)$ that dominates $\omega_{3-i}$.
	\end{itemize}
	More generally, for given sets $A,B\subset \hat{V}(G)$, a (graphic) $A-B$ \textbf{path} is any (graphic) $a-b$ path for some $a\in A$ and some $b\in B$. Additionally, only when $A, B\subset V(G)$, we impose that any $A-B$ path intersects $A\cup B$ precisely at their endpoints, which recovers the definition presented in the introduction. 
\end{definition}

Roughly speaking, the paths in Definition \ref{paths} differ from the graphic ones by allowing dominating vertices to represent some reachable end. If we consider only graphic paths, Bruhn, Diestel and Stein in \cite{versaodiestel} generalized Theorem \ref{erdosmenger} somehow verbatim, under a condition of topological separation:

\begin{theorem}[\cite{versaodiestel}, Theorem 1.1]\label{versaodiestel}
	In a given connected graph $G$, fix $A,B\subset \hat{V}(G)$ two sets that are separated topologically, i.e., such that $A\cap \overline{B} = \emptyset = \overline{A}\cap B$. Then, there exist a family $\mathcal{P}$ of disjoint graphic $A-B$ paths and a subset $S\subset \hat{V}(G)$ with the following properties:
	\begin{itemize}
		\item[$i)$] The graphic paths intersect $A$ and $B$ precisely at their endpoints. Formally, for every $P\in \mathcal{P}$ there are $a \in A$ and $b\in B$ such that $\overline{P} \cap (A\cup B) = \{a,b\}$;
		\item[$ii)$] Any two distinct graphic paths $P,Q\in\mathcal{P}$ are disjoint even at the infinity. In other words, $\overline{P}\cap \overline{Q} = \emptyset$;
		\item[$iii)$] $|\overline{P}\cap S| = 1$ for every $P\in \mathcal{P}$ and $S = \displaystyle \bigcup_{P\in\mathcal{P}}S\cap \overline{P}$;
		\item[$iv)$] If $Q$ is any graphic $A-B$ path, then $\overline{Q}\cap S \neq \emptyset$.
	\end{itemize}
\end{theorem}

Comparing Theorem \ref{versaodiestel} with the original statement of the Erd\H{o}s-Menger Theorem, the subset $S\subset \hat{V}(G)$ of the above thesis plays the role of an ``$A-B$ separator lying on the family $\mathcal{P}$''. Nevertheless, when applied to the graph $G$ of Figure \ref{raioduplodominado}, there are no two disjoint graphic paths connecting the only two ends $\omega_1,\omega_2\in \Omega(G)$. But, since $v_{\infty}$ is adjacent to every other vertex, at least two \textit{vertices} are needed to separate $\omega_1$ and $\omega_2$. The ``$\{\omega_1\}-\{\omega_2\}$ separator'' claimed by Theorem \ref{versaodiestel}, then, is given by either $\{\omega_1\}$ or $\{\omega_2\}$. However, if we aim to forbid ends in separating sets, the broader definition of connecting paths in Definition \ref{paths} - rather than only the graphic ones - is useful. Relying on this notion, Polat in \cite{polat} obtained the Mengerian result below:

\begin{theorem}[\cite{polat}, Theorem 3.4\footnote{When consulting the original statement in \cite{polat}, the reader may be aware that $A$ and $B$ are indeed end subsets of $G$, although a typo in Theorem 3.4 led to a notation for $\Omega(G)$ which disagrees with the one adopted in the whole paper.  }]\label{mengerpolat}
	Let $A,B\subset \Omega(G)$ be two sets of ends such that $\overline{A}\cap B = \emptyset  = A\cap \overline{B}$. Consider a family $\mathcal{P}$ of maximum size of disjoint $A-B$ paths and fix an $A-B$ separator $S$ of minimum size. Then, $|\mathcal{P}| = |S|$. 
\end{theorem}

Although we shall precise Polat's notion of ``$A-B$ separators'' only in the next section, we remark that Theorem \ref{mengerpolat} will be further stated in an edge-analogous setting as a consequence of our Theorem \ref{t2}. To this aim, however, the definitions employed in the above result need to be rewritten in their edge-related versions as well. Hence, inheriting the notion of $\mathcal{E}-$dominance from \cite{edgeends}, we say that a vertex $v\in V(G)$ of a graph $G$ \textbf{edge-dominates} an edge-end $\omega \in \Omega_E(G)$ if $v\in C_E(F,\omega)$ for every finite set $F\subset E(G)$. In this case, there is an infinite family of edge-disjoint paths connecting $v$ to infinitely many vertices of $r$. Then, we fix the following definitions:

\begin{definition}[Connecting paths - edge version]\label{epaths}
	Let $G$ be a graph. For vertex sets $A,B\subset V(G)$, the notion of an $A-B$ path given by Definition \ref{paths} is preserved. For connectivities between edge-ends and vertices and edge-ends, we consider the following criteria:
	\begin{itemize}
		\item \texttt{Connectivity between vertices and edge-ends:} Fix $v\in V(G)$ and $\omega \in \Omega_E(G)$. A \textbf{graphic} $v-\omega$ \textbf{path} is a ray starting at $v$ whose edge-end is $\omega$. Then, a $v-\omega$ \textbf{path} is either a graphic one or a $\{v\}-\{u\}$ path for some vertex $u\in V(G)$ that edge-dominates $\omega$;
		\item \texttt{Connectivity between edge-ends:} For edge-ends $\omega_1,\omega_2\in \Omega_E(G)$, an $\omega_1-\omega_2$ \textbf{path} is one of objects below:
		\begin{itemize}
			\item[$i)$] A $v-\omega_i$ path, for some $i\in \{1,2\}$ and some vertex $v\in V(G)$ that edge-dominates $\omega_{3-i}$; 
			\item[$ii)$] A double ray in which one half-ray has edge-end $\omega_1$ and the other has edge-end $\omega_2$. This case defines a \textbf{graphic} $\omega_1-\omega_2$ \textbf{path}.
		\end{itemize}
	\end{itemize}
	Finally, given $A,B\subset \hat{E}(G)$ such that $(A\cup B)\cap \Omega_E(G) \neq \emptyset$, an $A-B$ \textbf{path} is simply an $a-b$ path for some $a \in A$ and some $b\in B$.  
\end{definition}

In order to conclude a restricted version of Theorem \ref{t1}, which will be later useful for proving Theorem \ref{t2}, we turn our attention to subsets of $\hat{E}(G)$ that can be separated by finitely many edges. While discussing Menger-type results, we are particularly interested when these separators are small ones, often minimal or with minimum size. Then, in future arguments, we shall rely on the following remark:   

\begin{lemma}\label{cuts}
	Consider sets $A$ and $B$ such that either $A\subset V(G)$ or $A\subset \Omega_E(G)$, as well as either $B\subset V(G)$ or $B\subset \Omega_E(G)$. If it exists, let $F\subset E(G)$ be a $\subseteq-$minimal finite edge set that separates $A$ and $B$. Then, there is $X\subset V(G)$ such that $F = \delta(X)$. Besides that, $X$ can be chosen so that $A\cap \overline{C} \neq \emptyset$ for every connected component $C$ of $G[X]$.   
\end{lemma}
\begin{proof}
	Considering the topology of $\hat{E}(G)$, define the family $$\mathcal{C} = \{C: C \text{ is a connected component of }G\setminus F \text{ such that }A\cap \overline{C} \neq \emptyset\}.$$ 
	We will show that the claimed subset of $V(G)$ can be chosen as $X = \displaystyle \bigcup_{C\in \mathcal{C}}V(C)$. In fact, for every edge $e\in F$, there is an $A-B$ path $P$ in $G\setminus (F\setminus \{e\})$ by the minimality of $F$. This path contains $e$ because $F$ separates $A$ and $B$ by assumption, although $P$ contains no other edge from $F$. Therefore, $e$ has an endpoint in $X$ and the other in $V(G)\setminus X$, so that $e\in \delta(X)$.

	Conversely, by definition of $\mathcal{C}$, an edge $e\in \delta(X)$ has endpoints in two distinct connected components of $G\setminus F$. Hence, we must have $e\in F$.
\end{proof}

However, by possibly containing infinitely many edges, the precise definition of an $A-B$ separator as mentioned by Theorem \ref{t2} will be better introduced in Section \ref{s3}. From now until the end of the current one, we shall rather conclude the instance below and draw some related consequences. Roughly speaking, its proof follows from an iterative application of Corollary \ref{edgesmenger}, the Erd\H{o}s-Menger Theorem for edges. In the literature, its vertex-analogous statement is the Theorem 4.7 found in \cite{PolatFrances}, from where the details below are adapted: 

\begin{lemma}\label{vertex-end}
	Let $G$ be a connected graph. Fix an edge-end set $A\subset \Omega_E(G)$ and a vertex set $S\subseteq V(G)$. Assume the existence of a finite edge set $F\subset E(G)$ that separates $A$ and $S$, which we consider to have minimum size with that property. Then, there is a family $\mathcal{P}$ of $|F|$ many edge-disjoint $S-A$ paths.
\end{lemma}
\begin{proof}[Revisited proof of Theorem 4.7 in \cite{PolatFrances}]
	We first consider the case when $A$ is closed in $\Omega_E(G)$. Then, we will construct the required family $\mathcal{P}$ recursively, as the limit object from a sequence $(\mathcal{P}_n)_{n \in \mathbb{N}} = \{P_1^n,P_2^n,\dots,P_k^n\}_{n\in\mathbb{N}}$ of families of edge-disjoint finite paths. In this notation, $k := |F|$ and $P_i^{n+1}$ shall contain $P_i^n$ as an initial subpath, for each $n\in\mathbb{N}$ and $1 \leq i \leq k$.

	In order to initialize this construction, we rely on Lemma \ref{cuts} to write $F = \delta(V_0)$ for some $V_0\subset V(G)$, since $F$ is $\subseteq-$minimal regarding a separation property. By denoting $V_0' = V(G)\setminus V_0$, we assume that $S\subset V_0'$ and $A\subset \overline{V_0}$.

	Now, let $S'\subseteq V(G)$ comprise the endpoints of the edges in $F$ that belong to $V_0$. Denote by $\hat{G}$ the subgraph of $G$ obtained after adjoining to $G[V_0']$ precisely $S'$ and the edges from $F$. Hence, by the minimality of $F$, Corollary \ref{edgesmenger} (the Erd\H{o}s-Menger Theorem for edges) guarantees the existence in $\hat{G}$ of a family $\mathcal{P}$ of $k$ edge-disjoint $S-S'$ paths. In particular, each edge from $F$ belongs to precisely one path of $\mathcal{P}$.
	
	After setting $G_0 := G$, $S_0:= S$, $\mathcal{P}_0:=\mathcal{P}$ and $F_0:= F$, suppose by induction that we have defined the following objects and their properties:
	\begin{itemize}
		\item The family of paths $\mathcal{P}_n = \{P_i^n: 1 \leq i \leq k\}$. For each $1 \leq i \leq k$, denote by $v_i^n$ the endpoint of $P_i^n$ other than the one in  $S$;
		\item A subgraph $G_n$ of $G$ such that the cut $\delta(V(G_n))$ in $G$ is finite. Moreover, we assume that every edge-end of $A$ has a representative in $G_n$;
		\item A sizewise minimum cut $F_n$ in $G_n$ which separates a specified vertex set $S_n\subseteq V(G_n)$ from $A$. We also assume that the last edge from each path of $\mathcal{P}_n$ is an element of $F_n$. 
	\end{itemize}

	First, since $F_n$ is a cut of $G_n$, let us write $F_n = \delta(V_n)$ and $V_n':=V(G_n)\setminus V_n$ for some $V_n\subset V(G_n)$. We assume that $G_{n+1} := G_n[V_n]$ is the subgraph induced by the part of the bipartition $\{V_n, V_n'\}$ in which every end of $A$ has a representative.  Then, define $$S_{n+1} = \{v_j^n : \text{ finitely many edges separate }v_j^n\text{ from }A\text{ in }G_{n+1}\}.$$

	For every $1 \leq i \leq k$ such that $v_i^n \notin S_{n+1}$, we set $P_i^{n+1} = P_i^n$. Since $S_{n+1}$ is finite, its definition allows us to find a finite set $F_{n+1}\subset E(G_{n+1})$ that separates $S_{n+1}$ from $A$. By choosing $F_{n+1}$ of minimum size with that property, $F_{n+1}$ is a cut in $G_{n+1}$ by Lemma \ref{cuts}. Hence, we write  $F_{n+1} = \delta(V_{n+1})$ for some $V_{n+1}\subset V(G_{n+1})$, considering $V_{n+1}' := V(G_{n+1})\setminus V_{n+1}$ as the part of the bipartition $\{V_{n+1}, V_{n+1}'\}$ containing $S_{n+1}$.

	Denote by $K_n \subset F_n$ the set of edges of $F_n$ whose endpoints belong to $S_{n+1}$. Let $A'$ be the set of endpoints in $V_{n}'$ for the edges of $F_n$ and, similarly, let $B'$ be the set of endpoints in $V_{n+1}$ for the edges of $F_{n+1}$. Define $\hat{G}_{n}$ as the graph obtained after adding to $G_{n+1}[V_{n+1}']$ the vertices of $A'\cup B'$ and the edges of $K_n\cup F_{n+1}$. By Corollary \ref{edgesmenger} (the edge version of the Erd\H{o}s-Menger Theorem), there is a family $\mathcal{P}'$ of edge-disjoint $A'-B'$ paths in $\hat{G}_n$ together with a cut $C'$ obtained by the choice of precisely one edge from each path of $\mathcal{P}'$. We observe that $|C'|\geq |K_n|$. Otherwise, $F_n' = (F_n\setminus K_n)\cup C'$ has strictly fewer edges than $F_n$ and separates $A$ from $S_n$ in $G_n$, contradicting the definition of $F_n$. Therefore, each edge from $K_n$ lies on precisely one path of $\mathcal{P}'$. By concatenating these paths with the previous paths of $\mathcal{P}_n$ that end in $K_n \subset F_n$, we finish the definition of $\mathcal{P}_{n+1}$.

	At the end of this recursive process, consider $P_i' = \displaystyle \bigcup_{n\in\mathbb{N}}P_i^n$ for each $1 \leq i \leq k$, which is well-defined since, for every $n\in\mathbb{N}$, the path $P_i^{n+1}$ is constructed to contain $P_i^n$ as an initial segment. We observe that, if $P_i'$ is a ray, its edge-end belongs to $A$, so that $P_i'$ is an $S-A$ path. Otherwise, since $A$ is closed by assumption, there would be a finite set $F'\subset E(G)$ such that $A\cap \Omega_E(F',[P_i']_E) = \emptyset$. However, we can argue that $\displaystyle \bigcap_{n\in\mathbb{N}}E(G_n) = \emptyset$, proving that $F'\cap E(G_n) =\emptyset$ for some big enough $n$. In fact, suppose for a while that there exists an edge $e\in E(G)$ contained in $G_n$ for every $n\in\mathbb{N}$. Since $G$ is connected, there would be a (finite) path containing both $e$ and some edge from $F_0$. Nevertheless, this path should intersect $F_n$ for every $n\geq 1$, because $e\in E(G_n)$ and $F_0\subseteq E(G)\setminus E(G_{n})$. This reaches a contradiction once $\{F_n\}_{n\in\mathbb{N}}$ is a family of pairwise disjoint edge sets by construction, while the path just claimed is finite. Therefore, if $n\in\mathbb{N}$ is chosen so that $F'\cap E(G_n) = \emptyset$, then $C_E(F',[P_i']_E)$ contains the connected component of $G_n$ in which there is a tail of $P_i'$. This connected component also contains representatives of edge-ends in $A$ by Lemma \ref{cuts}, now contradicting the fact that $A\cap \Omega_E(F',[P_i']_E) = \emptyset$. Hence, $P_i'$ is indeed an $S-A$ path.

	Similarly, if $P_i'$ is a finite path for some $1 \leq i \leq k$, then $P_i' = P_i^n$ for all but finitely many numbers $n \in \mathbb{N}$. Moreover, by definition, $P_i'$ is also an $S-A$ path if, and only if, its endpoint $v_i$ edge-dominates an edge-end from $A$. Then, consider the set of indices $I = \{1 \leq i \leq k : P_i'\text{ is not a }S-A\text{ path}\}$. Let $J\subset I$ be a subset of maximum size for which there is a family $\mathcal{P}_J = \{P_j : j \in J\}$ of edge-disjoint $S-A$ paths satisfying the following properties:
	
	\begin{itemize}
		\item $\{P_i': i\notin J\}\cup \{P_j : j \in J\}$ is a set of edge-disjoint rays or finite paths;
		\item $P_j'$ is an initial subpath of $P_j$ for every $j\in J$.
	\end{itemize}

	If we prove that $J = I$, then $\mathcal{P} = \{P_i' : i \notin I\}\cup \{P_i : i \in I\}$ is the claimed family of edge-disjoint $S-A$ paths. Indeed, suppose that there exists $\ell\in I\setminus J$. Then, $P_{\ell}'$ is a finite path whose endpoint $v_{\ell}$ (other than the one in $S$) does not edge-dominate an end from $A$. In particular, $v_{\ell}$ does not edge-dominate an edge-end from $\{[P_i']_E : P_i' \text{ is a ray}, i\notin I\} \cup \{[P_j]_E : P_j\text{ is a ray}, j \in J\}$. Hence, there is a finite edge set $L\subseteq E(G)$ such that $C_E(L,v_{\ell})$ contains no tail from a ray in $\{P_i':i\notin J\}\cup \{P_j: j \in J\}$. This means that $C$ may contain edges from the finite paths in $\{P_i':i\notin J\}\cup \{P_j: j \in J\}$, but it contains only finitely many edges from the elements of this family which are rays. In other words, by possibly adding finitely many edges to $L$, we can assume that $E(C)$ actually does not intersect $\displaystyle \bigcup_{i\notin J}E(P_i')\cup \bigcup_{j\in J}E(P_j)$. However, by construction, $v_{\ell}$ cannot be separated from $A$ by finitely many edges, since $P_{\ell}'$ is a finite path. Hence, there is a ray in $C$ starting at $v_{\ell}'$ whose edge-end belongs to $A$. Concatenating this ray with the finite path $P_{\ell}'$, we contradict the maximality of $J$.

	Now, suppose that $A\subset \Omega_E(G)$ is any subset that can be separated from $S$ by the finite set $F\subset E(G)$. If $F$ has minimum size with that property, we will now finish the proof by reducing Lemma \ref{vertex-end} to the case just analyzed, in which $A$ was assumed to be closed. In that regard, we first observe that $F$ also separates $\overline{A}$ from $S$. Otherwise, suppose that $r$ is a graphic $S-\overline{A}$ path missing $F$. Since $[r]_E\in \overline{A}$, fix $\omega \in A\cap \Omega_E(F,[r]_E)$ and let $r'$ be a representative of this edge-end which starts in some vertex in $S$. Once $C_E(F,[r']_E) = C_E(F,[r]_E)$ by the choice of $\omega$, the tail of $r'$ in $G\setminus F$ is thus contained in $C_E(F,[r]_E)$. In other words, this tail of $r'$ lies in a connected component of $G\setminus F$ which contains a vertex of $S$, since $r$ is a $S-\overline{A}$ path that does not intersect $F$. However, this contradicts the choice of $F$ as an edge set that separates $S$ and $A$, because $[r']_E=\omega \in A$.

	Hence, by the case just analyzed, there is a family $\mathcal{P}$ of $k = |F|$ edge-disjoint $S-\overline{A}$ paths. Then, for every $P\in \mathcal{P}$ there is an edge-end $\omega_P\in \overline{A}$ such that $P$ is either a ray whose edge-end is $\omega_P$ or a finite path such that an endpoint edge-dominates $\omega_P$. In this latter case, by shortening the path $P$ if necessary, we assume that its endpoint is the unique vertex of $P$ that edge-dominates the edge-end $\omega_P$.

	Consider then the set $X = \{P\in \mathcal{P}:\omega_P \in \overline{A}\setminus A \}$. Suppose that $\mathcal{P}$ is chosen so that $X$ has minimum size. Finishing the proof, we claim that $X = \emptyset$. Otherwise, fix $P\in \mathcal{P}$ and, since $\mathcal{P}$ is finite, let $L\subset E(G)$ be a finite edge set such that $\omega_Q\notin \Omega_E(L,\omega_P)$ for every $Q\in \{Q'\in \mathcal{P}:\omega_{Q'}\neq \omega_P\}$. By possibly adding (finitely many) edges to $L$, we can even assume that no edge from $\displaystyle \bigcup_{\substack{Q\in \mathcal{P} \\\omega_Q \neq \omega_ P}}E(Q)$ lies in $C_E(L,\omega_P)$. However, there is a ray $r'$ in $C_E(L,\omega_P)$ whose edge-end $[r']_E$ belongs to $A$, since $\omega_P\in \overline{A}\setminus A$. By changing the choice of $P$ within $\{Q\in \mathcal{P}: \omega_Q = \omega_P\}$ if necessary, we can assume that $r'$ starts in a vertex of $P$ and does not intersect $\displaystyle \bigcup_{\substack{Q\in \mathcal{P} \\\omega_Q = \omega_ P}}V(Q)$ in any other point, because $[r']_E \neq \omega_P$. This defines an $S-A$ path $P'$ by concatenating $r'$ with an initial segment of $P$, so that $(\mathcal{P}\setminus\{P\})\cup \{P'\}$ contradicts the minimality of $X$.      
	
\end{proof}

Applying the above result to both sides of a bipartition of $V(G)$ given by a minimum $A-B$ cut, we obtain the following corollary as a particular case of Theorem \ref{t1}:

\begin{corollary} \label{finiteseparation}
	Given a graph $G$, let $A,B\subset \Omega_E(G)$ be two sets of edge-ends which can be separated by a finite edge set $F\subseteq E(G)$. If $F$ has minimum size with that property, there is a family $\mathcal{P}$ of $|F|-$many edge-disjoint $A-B$ paths.
\end{corollary}
\begin{proof}
	Since $F$ has minimum size, we can write $F = \delta(V_1)$ for some $V_1\subset V(G)$ by Lemma \ref{cuts}. Consider $V_2 = V(G)\setminus V_1$ and let $S_i$ be the set of endpoints of edges of $F$ in $V_i$, for $i=1,2$. Assume that the representatives of edge-ends in $A$ have their tails in $G[V_1]$, while the representatives of edge-ends in $B$ have their tails in $G[V_2]$. Hence, $F$ is an edge set of minimum size separating $A$ from $S_2$ in $G[V_1\cup S_2]$. By Lemma \ref{vertex-end}, there is a family $\mathcal{P}_1 = \{P_1^1,P_2^1,\dots,P_{|F|}^1\}$ of edge-disjoint $S_2-A$ paths. Analogously, there is a set $\mathcal{P}_2 = \{P_1^2,P_2^2,\dots,P_{|F|}^2\}$ of edge-disjoint $S_1-B$ paths in $G[V_2\cup S_1]$. By changing the enumeration of the elements in $\mathcal{P}_2$ if necessary, we observe that $P_i^1$ and $P_i^2$ intersect in precisely one edge from $F$ for every $1 \leq i \leq |F|$. Therefore, the concatenation $P_i^1P_i^2$ is a well defined $A-B$ path, so that $\mathcal{P} = \{P_i^1P_i^2 : 1 \leq i \leq |F|\}$ is the claimed family.
\end{proof}

Finishing this section, we will discuss how Lemma \ref{vertex-end} is useful for extending already known generalizations of the Lovász-Cherkassky Theorem, specially in order to draw a proof for Theorem \ref{t1} as one of our main results. Before that, we first recall its classical statement within finite graph theory:

\begin{theorem}[Lovász-Cherkassky]\label{lovaszfinito}
	In a given finite graph $G$, fix $T\subset V(G)$. If $G$ is inner-Eulerian for $T$, the maximum number of pairwise edge-disjoint $T-$paths in $G$ is equal to $$\frac{1}{2}\sum_{t\in T}\lambda(t,T\setminus \{t\}),$$ where $\lambda(t,T\setminus \{t\})$ denotes the minimum size of a cut separating $t$ from $T\setminus \{t\}$. 
\end{theorem}

In the above result, a $T-$\textbf{path} in $G$ is a $a-b$ path $P$ as in Definition \ref{paths}, where $a,b \in T$ are distinct endpoints. On the other hand, $G$ is \textbf{inner-Eulerian} for $T$ if every vertex of $V(G)\setminus T$ has even degree. Under this hypothesis, Theorem \ref{lovaszfinito} claims that there is a family $\mathcal{P}$ of edge-disjoint $T-$paths attaining an optimal size. After all, in a family of at least $1 + \displaystyle \frac{1}{2}\sum_{t\in T}\lambda(t,T\setminus \{t\})$ paths, two of them must intersect in an edge of a minimum cut separating some vertex $t$ from $T\setminus \{t\}$. In particular, for every $t\in T$, the maximum family $\mathcal{P}$ contains precisely $\lambda(t,T\setminus \{t\})$ paths that end at $t$. Considering this property, Jacobs, Joó, Knappe, Kurkofka and Melcher obtained a version of the Lovász-Cherkassky Theorem for locally finite graphs and their ends: 

\begin{theorem}[\cite{jacobs}, Theorem 1]\label{jooetal}
	Let $G$ be a locally finite graph and fix $T\subset \hat{V}(G) $ a discrete subset. Suppose that $|\delta(X)|$ is even or infinite for every $X\subset V(G)$ for which $T\subset \overline{X}$. Then, there is a family $\mathcal{P}$ of edge-disjoint graphic $T-$paths such that, for every $t\in T$, the number of graphic $\{t\}-(T\setminus \{t\})$ paths is equal to $\lambda(t,T\setminus \{t\})$. 
\end{theorem}

Following the notation from Definition \ref{paths}, a \textbf{graphic} $T-$\textbf{path} is a graphic $T-T$ path $P$ such that $|\overline{P}\cap T| = 2$. Intuitively, $P$ intersects $T$ precisely at its endpoints. Since Theorem \ref{jooetal} is restricted to locally finite graphs, we can rely on both Definitions \ref{paths} and \ref{epaths} to state the above result, as there is no distinction between ends and edge-ends in this graph family. Analogously, $\lambda(t,T\setminus \{t\})$ now denotes the minimum size of a $\{t\} - T\setminus\{t\}$ (edge-)separator, as in Definition \ref{eseparator}. By assuming that $T\subset \hat{V}(G)$ is discrete, $\lambda(t,T\setminus \{t\})$ is a well-defined natural number.

Comparing the Theorems \ref{lovaszfinito} and \ref{jooetal}, the ``inner-Eulerian'' assumption over $G$ is, in the locally finite generalization, replaced by a parity condition over some finite cuts. In the Appendix of \cite{jacobs}, Jacobs, Joó, Knappe, Kurkofka and Melcher discuss the need of this alternative notion in order to conclude Theorem \ref{jooetal}. Despite that, the new assumption indeed restricts to the original hypothesis for finite graphs. More precisely, if $G$ is finite and every vertex of $G\setminus T$ has finite degree, let $\tilde{G}$ be the multigraph obtained by contracting a set $X\supset T$ to a new vertex $v$. Since $\tilde{G}$ has an even number of vertices of odd degree, the degree of $v$ must be even, because so is the degree of every vertex from $V(\tilde{G})\setminus \{v\} = V(G)\setminus X$. Noticing that $|\delta(X)|$ is the degree of $v$ in $\tilde{G}$, the above result is, in fact, a generalization of the classical Lovász-Cherkassky Theorem for finite graphs. On the other hand, its proof relies on the countable version of Theorem \ref{lovaszfinito}, previously obtained by Joó in \cite{joo}. Below, we present an updated statement of Joó's result, established by him in \cite{joogeral} for arbitrary graphs:

\begin{theorem}[\cite{joogeral}, Theorem 1.2]\label{joo}
	Let $G$ be a (multi)graph and fix $T\subset V(G)$. Suppose that $|\delta(X)|$ is even or infinite for every $X\subset V(G)$ in which $T\subset X$. Then, there is a family $\mathcal{P}$ of edge-disjoint $T-$paths such that, for every $t\in T$, there is a cut separating $t$ from $T\setminus \{t\}$ obtained by the choice of precisely one edge from each path of $\mathcal{P}_t = \{P\in \mathcal{P}: t \text{ is an endpoint of }P\}$.
\end{theorem}

The proof of Theorem \ref{jooetal} is drawn by an appropriate reduction to Theorem \ref{joo}. In \cite{jacobs}, previously known results about the connectivity of locally finite graphs and its ends were also recalled for that study, such as Lemma 10 in \cite{finitos-um}. Now, in order to conclude the Lovász-Cherkassky Theorem for arbitrary infinite graphs and its edge ends, we can combine Lemma \ref{vertex-end} to the main idea for proving Theorem \ref{jooetal}. To this aim, a $T-$\textbf{path} in the result below means a $T-T$ \textbf{path} $P$ as in Definition \ref{epaths} and such that $|\overline{P} \cap T| = 2$: 

\begin{reptheorem}{t1}
	\textit{	Let $G = (V,E)$ be a graph and fix $T\subset \hat{E}(G)$ a discrete subspace of $\hat{E}(G)$. Suppose that $|\delta(X)|$ is even or infinite for every $X\subset V(G)$ in which $T\subset \overline{X}$. Then, there is a collection $\mathcal{P}$ of edge-disjoint $T-$paths such that, for every $t\in T$, there is a cut separating $t$ from $T\setminus \{t\}$ that lies on the family $\mathcal{P}_t = \{P\in \mathcal{P}: t \text{ is an endpoint of }P\}$.}
\end{reptheorem}    
\begin{proof}[Revisited proof of Theorem 1 in \cite{jacobs}]
	We shall rely on Lemma \ref{discreto}, which will be proven in the next section with the support of suitable tools that are developed there. Since $T\subset \hat{E}(G)$ is discrete, that result claims the existence of a family $\{C_{\omega}: \omega \in T\cap \Omega_E(G)\}$ of connected subgraphs of $G$ for which the following properties are verified:
	
	\begin{itemize}
		\item $C_{\omega_1}\cap C_{\omega_2} = \emptyset$ if $\omega_1 \neq \omega_2$;
		\item For every $\omega \in T\cap \Omega_E(G)$, the cut $F_{\omega}:=\delta(C_{\omega})$ is finite and $\hat{E}(F_{\omega},\omega)\cap T = \{\omega\}$.
	\end{itemize}

	Then, for every $\omega \in \Omega_E(G)$, let $\hat{C_{\omega}}$ denote the subgraph of $G$ induced by $C_{\omega}$ and the endpoints of the edges in the finite cut $F_{\omega} := \delta(C_{\omega})$. In particular, $F_{\omega}$ turns out to be a finite $N(C_{\omega})-\{\omega\}$ separator in $G$. By passing $C_{\omega}$ to a connected subgraph if necessary, Lemma \ref{cuts} allows us to assume that no edge set in $G$ with fewer than $|F_{\omega}|$ elements is also a $N(C_{\omega})-\{\omega\}$ separator. Hence, now according to Lemma \ref{vertex-end}, there is a family $\mathcal{P}_{\omega}$ of $|F_{\omega}|-$many edge-disjoint $N(C_{\omega})-\{\omega\}$ paths in $\hat{C_{\omega}}$.

	In its turn, let $\tilde{G}$ be the (multi)graph obtained from $G$ by contracting each $C_{\omega}$ to a vertex $v_{\omega}$. In this new graph, for every $\omega \in T\cap \Omega_E(G)$, the edges from $F_{\omega}$ are precisely the ones that are incident in $v_{\omega}$. Then, define also $T' = (T\cap V(G))\cup \{v_{\omega} : \omega \in T\cap \Omega_E(G)\}$ and let $X'\subseteq V(\hat{G})$ be any set of vertices containing $T'$. We observe that the edges in $\delta(X')\subseteq E(\hat{G})$ are precisely those in $G$ from the cut defined by $X:=\displaystyle (X'\cap V(G))\cup \bigcup_{\omega \in T\cap \Omega_E(G)}C_{\omega}$. Since $T\subseteq \overline{X}$, we know by assumption that this cut in $G$ is even or infinite, concluding that $\delta(X')$ is even or infinite. Therefore, the hypothesis of Theorem \ref{joo} (which has no mention of ends or edge-ends) are verified by $\hat{G}$ and $T'$. Then, let $\mathcal{P}'$ be a collection of edge-disjoint $T'-$paths such that, for every $t\in T'$, there is a cut $E_t$ which lies on the family $\mathcal{P}_t' = \{P\in \mathcal{P}': t \text{ is an endpoint of }P\}$ and separates $t$ from $T'\setminus \{t\}$.

	Finishing the proof, we describe the claimed collection $\mathcal{P}$ of edge-disjoint $T-$paths as follows: for every $\omega \in T\cap \Omega_E(G)$, we concatenate different paths from $\mathcal{P}_{v_{\omega}}'$ with different paths from the family $\mathcal{P}_{\omega}$. Note that $ |\mathcal{P}_{v_{\omega}}'| = |E_{v_{\omega}}|\leq |F_{\omega}| = |\mathcal{P}_{\omega}|$, because $|F_{\omega}|$ is precisely the degree of $v_{\omega}$ in $\tilde{G}$. Then, this construction of $\mathcal{P}$ is well defined, although some paths of $\mathcal{P}_{\omega}$ are not extended to paths of $\mathcal{P}$ if $|E_{v_{\omega}}|<|F_{\omega}|$.     
\end{proof}

\section{Topological and combinatorial separations}\label{s3}

We observe that our main results indicated in the introduction mention topological conditions on relevant sets of vertices or edge-ends. On one hand, the Lovász-Cherkassky theorem for infinite graphs requires that the set $T$ as in Theorem \ref{t1} is discrete. On the other, the sets $A$ and $B$ in Theorem \ref{t2} satisfy $\overline{A}\cap B = \emptyset = A\cap \overline{B}$. However, these hypotheses were not explored so far. In fact, Lemma \ref{vertex-end} was the core of the previous section, but it was obtained when supposing that its relevant objects could be separated by only finitely many edges. In its turn, this section shall explain how the topological assumptions on Theorems \ref{t1} and \ref{t2} ensure the well-definition of infinite edge sets as separators. In particular, we draw some combinatorial interpretations for standard notions within the topology of end and edge-end spaces.

Before that, it is useful to first recall some recent results towards the end structure of infinite graphs. For example, revisiting the work of Kurkofka and Pitz in \cite{representacao}, an \textbf{envelope} for a vertex subset $U\subset V(G)$ of a graph $G$ is a superset $U^*\supset U$ satisfying the two axioms below:

\begin{itemize}
	\item The neighborhood $N(C)\subseteq U^*$ is finite for every connected component $C$ of $G\setminus U^*$;
	\item $\overline{U^*}\setminus U^* = \overline{U}\setminus U$, i.e., the boundaries of $U$ and $U^*$ in the topological space $\hat{V}(G)$ define the same set of ends.
\end{itemize}

Envelopes for an arbitrary $U$ are constructed by Theorem 3.2 of \cite{representacao}, and we can even assume that they induce connected subgraphs. Moreover, the first condition above claims precisely that envelopes are subsets of \textbf{finite adhesion}, in the sense that the connected components of their complements have finite neighborhood. If $C$ is such a connected component and it contains a ray $r$, then $\hat{V}(N(C),[r])$ is a basic open set in $\hat{V}(G)$ around the end $[r]$. Due to this simple observation, most definitions and results in this section are statements about convenient subgraphs of finite adhesion.

Incidentally, rayless normal trees are probably the most well-known examples of these subgraphs. In fact, we recall that a tree $T$ in a graph $G$ is \textbf{normal} if, after fixing a tree-order $\leq$, any $V(T)-$path $P$ in $G$ has comparable endpoints regarding $\leq$. In particular, there is no edge in $G$ connecting incomparable elements of $T$ and, for every connected component $C$ of $G\setminus T$, its neighborhood $N(C)$ is totally ordered by $\leq$. In this case, $T$ has finite adhesion if it contains no rays. Considering that, Kurkofka, Melcher and Pitz in \cite{paracompacidade} discussed how rayless normal trees can approximate the end structure of a given graph:

\begin{theorem}[\cite{paracompacidade}, Theorem 1]\label{aproximacao}
	Let $G$ be an infinite graph. Fix an open cover for the subspace $\Omega(G)$ of the form $\mathcal{C} = \{\hat{V}(S_{\omega}, \omega): \omega \in \Omega(G)\}$. Then, there is a rayless normal tree in $G$ which refines $\mathcal{C}$. More precisely, for every connected component $C'$ of $G\setminus T$, there is $C \in \mathcal{C}$ such that $C'\subset C$.
\end{theorem}

Compiling many consequences of the above result, the paper \cite{paracompacidade} also contains unified proofs of previous works in infinite graph theory. Below, by relying on the notion of envelopes, we draw one more application of Theorem \ref{aproximacao}: 

\begin{lemma}\label{l1}
	For a fixed graph $G$, let $X\subset \hat{V}(G)$ be a discrete subspace. Then, there is a subgraph $H$ of $G$ containing $X\cap V(G)$ and for which the following properties hold:
	\begin{itemize}
		\item[$i)$] $H$ has finite adhesion;
		\item[$ii)$] For every two different ends $\omega_1,\omega_2 \in X$, there are distinct connected components $C_1$ and $C_2$ of $G\setminus H$ such that $\omega_1 \in \overline{C_1}$ and $\omega_2\in \overline{C_2}$. 
	\end{itemize}
\end{lemma}
\begin{proof}
	Applying Zorn's Lemma, let $\mathcal{W}$ be a $\subseteq-$maximal family of disjoint rays whose ends belong to $\overline{X}\setminus X$. Note that $\mathcal{W} = \emptyset$ if $X$ is closed. Then, consider a connected envelope $U^*$ for the vertex set $U := (X\cap V(G))\cup \displaystyle \bigcup_{s\in \mathcal{W}}V(s)$.

	We claim that $\overline{U^*}\setminus U^* = \overline{U}\setminus U = \overline{X}\setminus X$, where the first equality follows from the definition of envelopes. In fact, an element of $\overline{X}\setminus X$ is an end $\omega\in \overline{X}$. In this case, any of its representatives $r$ must intersect infinitely many vertices of $U$: otherwise, $\mathcal{W}\cup \{r'\}$ would contradict the $\subseteq-$maximality of $\mathcal{W}$, where $r'$ denotes a tail of $r$ such that $V(r')\cap U = \emptyset$. In particular, $C(S,\omega)\cap U \neq \emptyset$ for every finite set $S\subset V(G)$, proving that $\omega\in \overline{U}\setminus U$. Conversely, given $\omega\in \overline{U}\setminus U$, we must have $C(S,\omega)\cap U \neq \emptyset$ for every finite set $S\subset V(G)$. Since $\mathcal{W}$ is a family of disjoint rays, we may add finitely many vertices to a finite set $S\subset V(G)$ in order to assume that $s\setminus S$ is a tail of $s$ for every $s\in \mathcal{W}$. In this case, $C(S,\omega)\cap U$ must contain a vertex from $X\cap V(G)$ or the tail of a ray in $\mathcal{W}$, whose end belongs to $\overline{X}\setminus X$ by definition. In other words, the end $\omega$ lies in the closure of $(\overline{X}\setminus X)\cup (X\cap V(G))$, meaning that $\omega\in \overline{X}\setminus X$ since $X$ is discrete by assumption. Once showed that $\overline{U^*}\setminus U^* = \overline{X}\setminus X$, it follows that every end $\omega \in X$ has a representative in some connected component of $G\setminus U^*$. If not, we would have $C(S,\omega)\cap U^*\neq \emptyset$ for every finite set $S\subset V(G)$. This would mean that $\omega \in X\cap \overline{U^*}$, contradicting the fact that $\overline{U^*}\cap X \subseteq U^*$ contains no end.

	Now, let $D$ be a connected component of $G\setminus U^*$, whose neighborhood $N(D)$ is then finite. In particular, $\Omega(D) = \overline{D}\cap \Omega(G)$ is an open subspace of $\Omega(G)$. For every end $\omega \in \Omega(D)$, then either $\omega\in X$ or $\omega\notin \overline{X}$, since $\hat{V}(N(D),\omega)\cap U^* = \emptyset$ and $\overline{U^*}\setminus U^* = \overline{X}\setminus X$. In the first case, there is a finite set $S_{\omega}\subset V(G)$ such that $\hat{V}(S_{\omega},\omega)\cap X = \{\omega\}$, because $X$ is discrete. In the latter one, there is a finite set $S_{\omega}\subset V(G)$ for which $\hat{V}(S_{\omega},\omega)\cap X =\emptyset$. Clearly $\Omega(D)$ is thus covered by the family $\{\hat{V}(S_{\omega},\omega): \omega \in \Omega(D)\}$. Therefore, Theorem \ref{aproximacao} guarantees the existence of a rayless normal tree $T_D$ in $D$ with the following property: for every connected component $C$ of $D\setminus T_D$, there is an end $\omega\in \Omega(D)$ such that $C\subset \hat{V}(S_{\omega},\omega)$. In particular, by the choice of $S_{\omega}$, there is at most one element in $\overline{C}\cap X$. On the other hand, for every $\omega \in \omega(D)\cap X$ there is indeed a connected component $C_{\omega}^D$ of $D\setminus T_D$ such that $\omega \in \overline{C_{\omega}^D}$, once $T_D$ is a rayless normal tree. Then, consider the following family of connected subgraphs of $G$:
	
	$$\mathcal{K} = \{C_{\omega}^D: \omega \in \Omega(D)\cap X, D\text{ is a connected component of }G\setminus U^*\}.$$

	Finishing the proof, we will show that the graph $H$ claimed by the statement can be chosen as the one induced by $G\setminus \displaystyle \bigcup_{K\in \mathcal{K}}V(K)$. We first note that $H$ contains $U^*$, verifying the inclusion $X\cap V(G)\subset V(H)$. In addition, there is no edge in $H$ connecting distinct pairs $C_{\omega_1}^{D_1},C_{\omega_2}^{D_2}\in \mathcal{K}$. After all, $C_{\omega_1}^{D_1}$ and $C_{\omega_2}^{D_2}$ are contained in distinct connected components of $G\setminus U^*$ if $D_1 \neq D_2$. If $D: = D_1 = D_2$, then $C_{\omega_1}^{D_1}$ and $C_{\omega_2}^{D_2}$ are distinct connected components of $D\setminus T_D$, because $\omega_1,\omega_2 \in X\cap \Omega(D)$ are different ends. Thus, $\mathcal{K}$ is also the set of connected components of $G\setminus H$. Moreover, for every $\omega \in X$ there is a connected component $D$ of $G\setminus U^*$ in which $\omega$ has a representative, meaning that $\omega \in C_{\omega}^D$. To summarize, we just verified the second item of the statement. Finally, $H$ has finite adhesion since so do the envelope $U^*$ and the rayless normal trees of $\{T_D : D\text{ is a connected component of }G\setminus U^*\}$. In other words, a given component $C_{\omega}^D$ has its neighborhood contained in the finite set $N(D)\cup N_{T_D}(C_{\omega}^D)$.
\end{proof}

In a similar fashion, we remark that, if $X\subset \Omega(G)$ is a closed set of ends in a graph $G$, then there is also a subset $U^*\subset V(G)$ of finite adhesion satisfying $\overline{U^*}\setminus U^* = \overline{X}$. Indeed, we may consider $U^*$ as an envelope for the vertex set $U\subseteq V(G)$ of a $\subseteq-$maximal disjoint family of rays whose end belong to $X$, observing that $\overline{U}\setminus U=\overline{X} = X$ by the same reasoning applied in the second paragraph of the above proof. If we now write $X = \overline{A}$ for a given end set $A\subset \Omega(G)$ and we fix $B\subset \Omega(G)$ such that $\overline{A}\cap B = \emptyset$, then every representative $r$ of an end in $B$ has a tail in a connected component $C_{[r]}$ of $G\setminus U^*$. When we further assume that $A\cap \overline{B} = \emptyset$, then an envelope for the set $\displaystyle \bigcup_{\omega \in B}N(C_{\omega})$ plays the role of a graph $H$ as in the statement below:

\begin{lemma}[\cite{polat}, Theorem 1.2]\label{l2}
	Let $A,B\subset \Omega(G)$ be two sets of end in a graph $G$ which are topologically separated, i.e., such that $\overline{A}\cap B = \emptyset = A \cap \overline{B}$. Then, there is a subgraph $H$ of $G$ satisfying the two properties below:
	\begin{itemize}
		\item[$i)$] $H$ has finite adhesion and $\overline{H}\cap (A\cup B) = \emptyset$;
		\item[$ii)$] Given two rays $r$ and $s$ in $G$ such that $[r]\in A$ and $[s]\in B$, their tails in $G\setminus H$ are contained in distinct connected components.
	\end{itemize}
\end{lemma}

The subgraph $H$ in the above result may be understood as a (possibly infinite) structure that separates $A$ and $B$ combinatorially. Its existence was first established by Polat in \cite{polat}, who relied in his theory of \textit{multiendings} to that aim. In his work, especially when stating Theorem \ref{mengerpolat}, $H$ is an example of the objects he introduces as $A-B$ separators. In fact, Polat's definition of separators reads precisely as follows:

\begin{definition}[\cite{polat}, Definition 1.1]\label{separator}
	Fix a graph $G$ and two end subsets $A,B\subset \Omega(G)$. Then, $S\subset V(G)$ is called an $A-B$ \textbf{separator} if $\overline{S}\cap (A\cup B) = \emptyset$ and, for every connected component $C$ of $G\setminus S$, we have $\overline{C}\cap A = \emptyset$ or $\overline{C}\cap B = \emptyset$. This requires $A$ and $B$ to be disjoint, since every end from $A\cup B$ has a representative ray in $G\setminus S$. Moreover, $G\setminus S$ must contain no $A-B$ path as in Definition \ref{paths}.     
\end{definition}

Incidentally, Theorem 1.2 in \cite{polat} claims also the converse of Lemma \ref{l2}: i.e., there is an $A-B$ separator in a graph $G$ if, and only if, $A,B\subseteq \Omega(G)$ satisfy $\overline{A}\cap B = \emptyset = A\cap \overline{B}$. However, in this paper we are mainly interested in edge-connectivity results. Therefore, we shall now restate the conclusions of Lemmas \ref{l1} and \ref{l2} for studying suitable sets of edge-ends rather than subspaces of $\Omega(G)$. As a familiar tool around these discussions, we recall that the \textbf{line graph} $G'$ of a given graph $G$ is defined by setting $V(G') = E(G)$ and 
\begin{center}
	$ef \in E(G')$ if, and only if, $e$ and $f$ are \textbf{adjacent edges},
\end{center} namely, $e$ and $f$ share a common endpoint.

The following observation, then, easily translates paths in $G$ to paths in $G'$ and conversely:

\begin{lemma}\label{41}
	Fix $e$ and $f$ two non-adjacent edges in a connected graph $G$. Therefore, the two statements below hold:
	\begin{itemize}
		\item[i)] If $v_0v_1v_2\dots v_n$ is a path in $G$ such that $e = v_0v_1$ and $f = v_{n-1}v_n$, then $(v_{i}v_{i+1})_{i < n}$ defines a path in $G'$ connecting $e$ and $f$;
		\item[ii)] Conversely, if $e_0e_1e_2\dots e_n$ is a vertex-minimal path in $G'$ connecting the edges $e = e_0$ and $f = f_n$ of $G$, then there is a path $v_0v_1v_2\dots v_{n+1}$ in $G$ such that $e_i = v_iv_{i+1}$ for each $0 \leq i \leq n$.  
	\end{itemize}
\end{lemma}
\begin{proof}
	The first item is trivial. In fact, if $v_0v_1v_2\dots v_n$ is a path in $G$, then $v_{i_1}v_i$ and $v_iv_{i+1}$ are distinct and adjacent edges for every $1 < i < n$, defining a path in $G'$.

	Now, suppose that $e_0e_1e_2\dots e_n$ is a minimal path in $G'$ connecting the edges $e = e_0$ and $f = f_n$. Hence, for each $i < n$, $e_i$ and $e_{i+1}$ share an endpoint $v_{i+1}$ in $G$, since those are adjacent edges. Moreover, $v_{i+1}\neq v_{j+1}$ for every $j < n$ distinct from $i$. Otherwise, if $v_{i+1} = v_{j+1}$ and $i < j$ for instance, the edges $e_i$ and $e_{j+1}$ would be adjacent, so that the path $e_0e_1 \dots e_i e_{j+1}\dots e_n$ would contradict the minimality of $e_0e_1e_2\dots e_n$.   
\end{proof}

Besides that, the first item in the above lemma also provides a natural identification of some ends of $\Omega_E(G)$ with elements of $\Omega(G')$. More precisely, given a ray $r = v_0v_1v_2\dots$, we denote by $\phi(r)$ the ray in $G'$ presented by $\phi(r) = \{v_iv_{i+1}\}_{i < \omega}$. We thus consider the induced map  
\begin{equation}\label{phi}
	\begin{array}{cccc}
		\Phi:  &  \Omega_E(G) & \to & \Omega(G')\\
		& [r]_E & \mapsto & [\phi(r)]
	\end{array} 
\end{equation}

It is not difficult to prove that $\Phi$ is a well-defined injection, as it is also claimed by Lemma 1 of \cite{edgeends}. However, one cannot ensure that $\Phi$ is surjective, since infinitely many edges incident to a vertex $v$ of $G$ may define a ray in $G'$ whose end belongs to $\Omega(G')\setminus \Omega_E(G)$. In other words, besides being a natural inclusion of $\Omega_E(G)$ in $\Omega(G')$, the map $\Phi$ might not identify these two spaces. Thus, unlike in finite graph theory (where many results of edge-connectivity are restrictions of vertex-analogous statements to line graphs), the interplay between $G$ and $G'$ needs to be analyzed with additional caution if they contain vertices of infinite degree. Considering that, the edge-analogous of Lemma \ref{l1} reads as follows:

\begin{lemma}\label{discreto}
	For a fixed graph $G$, let $T\subset \hat{E}(G)$ be a discrete subset. Then, there is a family $\{C_{\omega}:\omega \in T\cap \Omega_E(G)\}$ of connected subgraphs of $G$ for which the following properties are verified:
	\begin{itemize}
		\item[$i)$] $C_{\omega_1}\cap C_{\omega_2} = \emptyset$ if $\omega_1 \neq \omega_2$;
		\item[$ii)$] For every $\omega \in T\cap \Omega_E(G)$, the cut $F_{\omega}:=\delta(C_{\omega})$ is finite and $\hat{E}(F_{\omega},\omega)\cap T = \{\omega\}$.
	\end{itemize}
\end{lemma}
\begin{proof}
	For every $v\in T\cap V(G)$, let $K_v = \delta(\{v\})$ denote the set of edges in $G$ which have $v$ as an endpoint. Therefore, $K_v$ is a vertex set that induces a clique in the line graph $G'$. We will then show that the set $T' =\displaystyle  \Phi(T\cap \Omega_E(G))\cup \bigcup_{v \in T\cap V(G)}K_v$ is also discrete in $\hat{V}(G')$, where $\Phi$ is the map given by (\ref{phi}). Since vertices are isolated points in this topological space, it is sufficient to conclude that there is no edge-end $[r]_E\in T\cap \Omega_E(G)$ such that $[\phi(r)]\in \overline{T'\setminus \{[\phi(r)]\}}$. In fact, for every edge-end $[r]_E \in T\cap \Omega_E(G)$, there is a finite set $F\subset E(G)$ such that $\hat{E}(F,[r]_E)\cap T = \{[r]_E\}$, because $T$ is a discrete subspace of $\hat{E}(G)$. Regarding $F$ as a finite vertex subset of $G'$, we must have $\hat{V}(F,[\phi(r)])\cap T' = \{[\phi(r)]\}$. Otherwise, one of the following items is verified:
	\begin{itemize}
		\item If there is $e \in \hat{V}(F,[\phi(r)])\cap T'\cap V(G')$, then $e$ is an edge in $G$ which is incident to some vertex $v\in T$. In this case, by Lemma \ref{41}, any vertex-minimal path in $G'\setminus F$ connecting $e$ to a tail of $\phi(r)$ would describe a path in $G$ connecting $v$ to a tail of $r$, but avoiding the edges in $F$. This contradicts the fact that $v\notin \hat{E}(F,[r]_E)\cap T$;
		\item If there is an end $\omega \in \hat{V}(F,[\phi(r)])\cap T'\cap \Omega(G')$, then $\omega = [\phi(s)]$ for some ray $s$ in $G$ such that $[s]_E\in T$. In this case, again by Lemma \ref{41}, any vertex-minimal path in $G'\setminus F$ connecting the tails of $\phi(r)$ and $\phi(s)$ would describe a path in $G$ connecting the tails of $r$ and $s$, but avoiding the edges of $F$. Since $\hat{E}(F,[r]_E)\cap T = \{[r]_E\}$, we then must have $[r]_E = [s]_E$. Therefore, $\omega = [\phi(r)]$. 
	\end{itemize}
	Hence, $T'$ is indeed a discrete subspace of $\hat{V}(G')$. Then, let $H$ be the subgraph of $G'$ claimed by Lemma \ref{l1} when considering $X = T'$. Therefore, for each edge-end $\omega \in T\cap \Omega_E(G)$ there is a connected component $K_{\omega}$ of $G'\setminus H$ such that $\Phi(\omega)\in \overline{K_{\omega}}$. Observing that $V(K_{\omega})\subset V(G') = E(G)$, we define $C_{\omega}$ as the subgraph of $G$ induced by the vertices which are endpoints of some edge from $V(K_{\omega})$. In particular, $C_{\omega}$ is connected by Lemma \ref{41} and must contain all but finitely many edges of any ray $r$ of $G$ such that $\omega = [r]_E$, since a tail of $\phi(r)$ in $G'$ is contained in $K_{\omega}$ in this case.

	We now note that $C_{\omega_1}\cap C_{\omega_2} = \emptyset$ if $\omega_1 \neq \omega_2$. In fact, if there were a vertex $v \in C_{\omega_1}\cap C_{\omega_2}$, then $v$ would be a common endpoint of some given edges $e \in V(K_{\omega_1})$ and $f \in V(K_{\omega_2})$. This means that $e$ and $f$ would be adjacent edges, belonging (as vertices) to the same connected component $K_{\omega_1} = K_{\omega_2}$ of $G'\setminus H$. However, this would contradict the injectivity of $\Phi$, once Lemma \ref{l1} asserts that $K_{\omega_1} = K_{\omega_2}$ if, and only if, $\Phi(\omega_1) = \Phi(\omega_2)$.

	Finally, for a fixed $\omega \in T\cap \Omega_E(G)$, we observe that the following equivalences hold for any edge $e\in E(G)$:
	
	\begin{align*}
		e \in \delta(C_{\omega}) & \iff e=xy \text{ for some }x\in C_{\omega}\text{ and some }y\in V(G)\setminus C_{\omega} \\
		& \iff e=xy\text{ for some endpoint }x\text{ of an edge }f\in V(K_{\omega})\text{ and some }y \notin V(C_{\omega}) \\
		& \iff e \text{ is adjacent to some }f\in V(K_{\omega})\text{ but do not belong to }K_{\omega}\\
		& \iff e \text{ belongs to }N(K_{\omega}) :=\{h \in H : h\text{ has a neighbor in }K_{\omega}\} \text{ as a vertex of }G'.
	\end{align*}

	In other words, we just verified the equality $\delta(C_{\omega}) = N(K_{\omega})$. From the finite adhesion of $H$, it follows that $F_{\omega}:=\delta(C_{\omega})$ is finite. Hence, it remains to prove that $\hat{E}(F_{\omega},\omega)\cap T \cap V(G) = \emptyset$, once we already concluded that $\omega \in \overline{C_{\omega}} \subset \hat{E}(F_{\omega},\omega)$ and that $C_{\omega}\cap C_{\eta} = \emptyset$ if $\eta \in \Omega_E(G) \cap T\setminus\{\omega\}$. For instance, thus, suppose that there is a vertex $v \in \hat{E}(F_{\omega},\omega)\cap T \cap V(G)$. Then, there is a path $P$ in $C_{\omega}$ connecting $v$ to a tail of a ray $r$ such that $\omega = [r]_E$. In this case, Lemma \ref{41} claims that the edges of $P$ define a path in $G'$ which connects an element from $K_v = \delta(\{v\})$ to a tail of $\phi(r)$. However, contradicting the fact that $E(P)\subset E(C_{\omega})$, this path must intersect $N(K_{\omega}) = \delta(C_{\omega})$, because $K_v\subseteq T'$ and $[\phi(r)] = \Phi(\omega) \in \overline{K_{\omega}}$. 
\end{proof}

Revisiting the previous section, we observe that Lemma \ref{discreto} played a central role in the proof of Theorem \ref{t1}. In addition, from the first paragraphs of the above proof we can extract the following conclusion: if $[\phi(r)]\in \overline{\Phi(T)}$ for some edge-end $[r]_E\in \Omega_E(G)$ and some set $T\subset \Omega_E(G)$, then $[r]_E \in \overline{T}$. In other words, the function $\Phi$ as defined in (\ref{phi}) is an open map over its range. Although simple, this remark is useful for translating Lemma \ref{l2} to its edge-related version:

\begin{lemma}\label{aresta-separador}
	Let $A,B\subset \Omega_E(G)$ be two sets of edge-ends in a graph $G$ which are topologically separated, i.e., such that $\overline{A}\cap B = A\cap \overline{B} = \emptyset$. Then, there is an edge set $H\subset E(G)$ with the following properties:
	\begin{itemize}
		\item[$i)$] For every edge-end $\omega \in  A \cup B$, there is a connected component of $G\setminus H$ which assumes the form $C_E(F_{\omega},\omega)$ for some finite edge set $F_{\omega}\subseteq E(G)$;
		\item[$ii)$] For given rays $r$ and $s$ in $G$ such that $[r]_E\in A$ and $[s]_E \in B$, their tails are contained in distinct connected components of $G\setminus H$.  
	\end{itemize}
\end{lemma}
\begin{proof}
	Let $\Phi$ be the function defined as in (\ref{phi}), which is an open map over its range. Therefore, $\overline{\Phi(A)}\cap \Phi(B) = \Phi(A)\cap \overline{\Phi(B)} = \emptyset$ by the main hypothesis over the sets $A$ and $B$. If $G'$ denotes the line graph of $G$, let $H$ be its subgraph of finite adhesion claimed by Lemma \ref{l2}. Then, $\overline{H}\cap (\Phi(A)\cup \Phi(B)) = \emptyset$ and, given two rays $r$ and $s$ in $G$ such that $[r]_E \in A$ and $[s]_E \in B$, the tails of $\phi(r)$ and $\phi(s)$ in are contained in distinct connected components of $G'\setminus H$. 
	
	In particular, for an edge-end $\omega \in A\cup B$ and a representative ray $r\in \omega$, the connected component $K_{\omega}$ of $G'\setminus H$ containing a tail of $\phi(r)$ has finite neighborhood. In other words, the set $N(K_{\omega}) = \{h \in H : h \text{ has a neighbor in }K_{\omega}\}$ is finite. Nevertheless, we can see $F_{\omega}:=N(K_{\omega})$ as a finite edge set in $G$, once $V(H)\subset E(G)$ by the definition of line graph. In addition, there is no edge $h\in V(H)$ with both endpoints in $C_E(F_{\omega},\omega)$: otherwise, there would exist in this subgraph of $G$ a path $P$ connecting a tail of $r$ to an endpoint of $h$ (which is an edge in $G$). However, this would imply by Lemma \ref{41} in the existence of a path $Q$ in $G'\setminus N(K_{\omega})$ connecting a tail of $\phi(r)$ to $h\in V(H)$, contradicting the fact that $K_{\omega}$ is the connected component of $G'\setminus H$ in which $\phi(r)$ has its tail. In particular, we just proved that $C_E(F_{\omega},\omega)$ is a connected subgraph of $G\setminus H$. Actually, observing that $F_{\omega}\subseteq H$ by definition, $C_{E}(F_{\omega},\omega)$ is indeed a connected component of $G\setminus H$.  Hence, the edge set $V(H)\subset E(G)$ verifies the first property required by the statement.

	Therefore, we shall finish the proof by arguing that $V(H)$ fulfills the property $ii)$ as well. To this aim, fix two rays $r$ and $s$ in $G$ such that $[r]_E \in A$ and $[s]_E \in B$. Suppose for a contradiction that their tails are contained in a same connected component of $G\setminus V(H)$, meaning that there is a path in $G$ connecting these tails but avoiding the edges from $V(H)$. However, by Lemma \ref{41}, this would define a path in $G'\setminus H$ which connects the rays $\phi(r)$ and $\phi(s)$, contradicting the second item of Lemma \ref{l2} for the choice of $H$.
\end{proof}

In its turn, the two properties of the edge set $H$ claimed by Lemma \ref{aresta-separador} inspire an edge-related version of Definition \ref{separator}. More precisely, under the assumption that $A,B\subset \Omega_E(G)$ are topologically separated as in the above result, we just verified that the combinatorial structure below indeed exists within the graph $G$:

\begin{definition}[Separators - edge version]\label{eseparator}
	Fix a graph $G$ and two end subsets $A,B\subseteq \Omega_E(G)$. Then, we call $S\subset E(G)$ an $A-B$ \textbf{separator} if there is no graphic $A-B$ path in $G\setminus S$ and, for every edge-end $\omega \in A\cup B$, there exists a finite set $F\subset E(G)$ such that no edge from $S$ has both its endpoints in $C_E(\omega, F)$. In particular, every edge-end from $A\cup B$ has a representative ray in $G\setminus S$, but this subgraph contains no $A-B$ path as in Definition \ref{epaths}.  
\end{definition}

Therefore, the settings of Theorem \ref{t2} were finally fully discussed: for given sets $A,B\subset \Omega_E(G)$, the $A-B$ paths are those introduced by Definition \ref{epaths}, while the $A-B$ separators are the ones as in Definition \ref{eseparator}. Moreover, Corollary \ref{finiteseparation} already established this Menger-type result precisely when $A$ and $B$ can be separated by finitely many edges. The general case, then, is covered by the proof below:

\begin{reptheorem}{t2}[Erd\H{o}s-Menger Theorem for edge-ends]\label{edge-ends}
	\textit{Let $A,B \subset \Omega_E(G)$ be two set of ends such that $\overline{A}\cap B = A\cap \overline{B} = \emptyset$. Then, there is a family $\mathcal{P}$ of edge-disjoint $A-B$ paths and an $A-B$ separator $S\subseteq E(G)$ which lies on it. }
\end{reptheorem} 
\begin{proof}
	Since $\overline{A}\cap B = \emptyset =A\cap \overline{B}$, fix an edge set $H'\subseteq E(G)$ satisfying properties $i)$ and $ii)$ of Lemma \ref{aresta-separador}. Then, define the family $\mathcal{C}' = \{C': C'\text{ is a connected component of }G\setminus H' \text{ such that }\overline{C'}\cap (A\cup B) \neq \emptyset\}$. By property $i)$, the cut $\delta(C')$ is finite for every $C'\in \mathcal{C}'$. In particular, $\delta(C')$ is a finite edge set that separates $N(C')$ from $\overline{C}'\cap (A\cup B)$ in the graph induced by $V(C')\cup N(C')$. Hence, by Lemma \ref{cuts}, there is a vertex set $X_{C'}\subseteq V(C')$ such that $\delta(X_{C'})$ is a sizewise minimum finite cut which separates $N(C')$ and $\overline{C}'$ in $G[V(C')\cup N(C')]$. Moreover, $X_{C'}$ can be chosen so that every connected component of $G[X_{C'}]$ contains a representative ray of some edge-end in $A\cup B$. Therefore, after writing $H_{C'}$ for the edge set of the graph induced by $V(C')\setminus X_{C'}$, the edge set $H:=H'\cup \displaystyle \bigcup_{C'\in\mathcal{C}'}(H_{C'}\cup \delta(X_{C'}))$ also satisfies properties $i)$ and $ii)$ of Lemma \ref{aresta-separador}. Besides that, the following additional assertion now holds:
	
	\begin{center}
		\texttt{Claim:} Let $C$ be a connected component of $G\setminus H$ such that $\overline{C}\cap (A\cup B)\neq \emptyset$. Then, $\delta(C)$ is a sizewise minimum cut that separates $N(C)$ and $\overline{C}\cap (A\cup B)$ in $G$.
	\end{center} 
	\begin{proof}[Proof of the claim]
		Suppose that there is a cut $F_C\subseteq E(C)\cup \delta(C)$ which separates $N(C)$ and $\overline{C}\cap (A\cup B)$ but such that $|F_C|<|\delta(C)|$. Then, let $C'$ be the connected component of $G\setminus H'$ which contains $C$ as a subgraph. By definition of $H$, note that $C$ is a connected component of $G[X_{C'}]$. In this case, $(\delta(X_{C'})\setminus \delta(C))\cup F_C$ is an edge set that separates $N(C')$ from $\overline{C'}\cap (A\cup B)$, but it has strictly fewer than $|\delta(X_{C'})|$ many edges. This contradicts the choice of $X_{C'}$, concluding that $|F_C| = |\delta(C)|$. 
	\end{proof}

	Now, define $\mathcal{C} = \{C: C\text{ is a connected component of }G\setminus H \text{ such that }\overline{C}\cap (A\cup B) \neq \emptyset\}$. Then, for every $C\in\mathcal{C}$, Lemma \ref{vertex-end} ensures the existence of a family $\mathcal{P}_C$ of $|\delta(C)|$ many edge-disjoint $N(C)-(\overline{C}\cap (A\cup B))$ paths in the graph induced by $V(C)\cup N(C)$. On the other hand, let $\tilde{H}$ denote the subgraph of $G$ comprising the edges from $H \cup \displaystyle \bigcup_{C\in\mathcal{C}}\delta(C)$ and from the connected components of $G\setminus H$ which have no representatives of edge-ends in $A\cup B$. Thus, we fix the sets $\tilde{A} = \{v \in \tilde{H}: v \in V(C)\text{ for some }C\in\mathcal{C}\text{ such that }\overline{C}\cap A \neq \emptyset\}$ and $\tilde{B} = \{v \in \tilde{H}: v \in V(C)\text{ for some }C\in\mathcal{C}\text{ such that }\overline{C}\cap B \neq \emptyset\}$, which are disjoint by property $ii)$ of Lemma \ref{aresta-separador}.

	Therefore, by the Erd\H{o}s-Menger Theorem for edges (Corollary \ref{edgesmenger}), there exist in $\tilde{H}$ a family $\mathcal{P}$ of edge-disjoint $\tilde{A}-\tilde{B}$ paths and a cut $S:=\delta(X)$ lying on it. Moreover, we choose $X\subseteq V(H)$ so that $\tilde{A}\subseteq X$ and $\tilde{B}\subseteq V(G)\setminus X$. When writing a given path $P\in \mathcal{P}$ as $P = v_0v_1\dots v_{n-1}v_{n}$, we thus can assume that $v_0\in \tilde{A}$ and $v_n\in \tilde{B}$. By definition of $\tilde{A}$ and $\tilde{B}$, there are connected components $C$ and $D$ of $G\setminus H$ such that $v_0\in V(C)$, $\overline{C}\cap A \neq \emptyset$, $v_n \in V(D)$ and $\overline{D}\cap B\neq \emptyset$. Then, the edges $v_0v_1 \in \delta(C_0)$ and $v_{n-1}v_n \in \delta(C_n)$ belong to certain (possibly infinite) paths $Q\in \mathcal{P}_{C}$ and $Q'\in \mathcal{P}_{D}$, respectively. Note that $Q$ and $Q'$ are uniquely determined as paths containing $v_0v_1$ and $v_{n-1}v_n$ respectively, since the families $\mathcal{P}_C$ and $\mathcal{P}_D$ comprise pairwise edge-disjoint (possibly infinite) paths. By concatenating $Q$, $P$ and $Q'$ accordingly, we then uniquely describe an $A-B$ path $\tilde{P}$. Regarding this construction, which is similar to the one in the last paragraph within the proof of Theorem \ref{t1}, $\tilde{\mathcal{P}} := \{\tilde{P}:P\in\mathcal{P}\}$ is a well-defined family of edge-disjoint $A-B$ paths that contains the cut $S$ on it. By properties $i)$ and $ii)$ of Lemma \ref{aresta-separador}, that are satisfied by $H$, every $A-B$ path in $G$ must contain the edges of an $\tilde{A}-\tilde{B}$ path in $\tilde{H}$ and, hence, must intersect the $\tilde{A}-\tilde{B}$ separator $S$. Observing that $S\subseteq \tilde{H}$, this establishes $S$ also as an $A-B$ separator and finishes the proof.          
\end{proof}

\begin{corollary}\label{antigot2}
	Let $A,B\subseteq \Omega_E(G)$ be two sets of edge-ends such that $\overline{A}\cap B = \emptyset = A\cap \overline{B}$. Fix a sizewise maximum family $\mathcal{P}$ of edge-disjoint $A-B$ paths and let $S\subseteq E(G)$ be an $A-B$ separator of minimum size. Then, $|\mathcal{P}| = |S|$. 
\end{corollary}
\begin{proof}
	Let $\mathcal{P}$ be the family claimed by the previous theorem and $S$ be the corresponding $A-B$ separator. Once $S$ lies on $\mathcal{P}$, we clearly have $|S| = |\mathcal{P}|$. In particular, $S$ is a sizewise maximum $A-B$ separator: after all, if $S'\subseteq E(G)$ has fewer than $|S|$ many edges, then there is $P\in\mathcal{P}$ such that $E(P)\cap S'=\emptyset$. On the other hand, since $S$ must meet every $A-B$ path, it follows that $\mathcal{P}$ has maximum size as a family of edge-disjoint $A-B$ paths. 
\end{proof}

In particular, the above corollary translates Theorem \ref{mengerpolat} verbatim to an edge-related context. Despite that, it lacks the structural property highlighted by Theorem \ref{t2}, which is more general and had no explicit vertex-analogous counterpart in the literature regarding end spaces. In fact, its proof here is supported by the Erd\H{o}s-Menger Theorem (for edges), a result due to Aharoni and Berger in \cite{erdosmenger} that was not available at the time Polat published his Theorem \ref{mengerpolat} in \cite{polat}.    

\section{Acknowledgments}

We thank the support of Fundação de Amparo à Pesquisa do Estado de São Paulo (FAPESP). The first and the second named authors were sponsored through grant numbers 2023/00595-6 and 2021/13373-6 respectively. We also acknowledge the referees by the careful analysis of this work, whose corresponding reports led to many improvements throughout the paper. In particular, we thank to the reviewer that asked for a generalization of Corollary \ref{antigot2}, which was written as one of our main results in a previous version. Obtained from this instigation, Theorem \ref{t2} now fits as a more structural statement regarding edge-connectivity between edge-ends in infinite graphs.

\bibliography{EdgeConnectivityBetweenEdgeEnds}
\bibliographystyle{plain}

\end{document}